\documentclass[12pt]{amsart}
\usepackage{amsmath,amssymb,amsfonts,amsthm,amscd,textcomp,times,booktabs,mathrsfs}
\usepackage[dvips]{graphicx}
\usepackage{braket}
\usepackage{color,float}
\usepackage{bm}
\usepackage[all]{xy}
\usepackage{enumerate}
\usepackage{mathtools}
\usepackage[normalem]{ulem}

\newtheorem{theorem}{Theorem}

\newtheorem{proposition}[theorem]{Proposition}
\newtheorem{lemma}[theorem]{Lemma}
\newtheorem{definition}[theorem]{Definition}
\newtheorem{corollary}[theorem]{Corollary}
\newtheorem{remark}[theorem]{Remark}

\newtheorem{step}{Step}

\numberwithin{equation}{section}
\numberwithin{theorem}{section}
\theoremstyle{definition}

\newcommand{\bs}{\boldsymbol} 

\newcommand{\din}{\rotatebox{90}{\ensuremath{\in}}}

\newcommand{\reduce}[1]{\scalebox{1}{\ensuremath{#1}}}

\newcommand{\mapright}[1]{\smash{\mathop{   \hbox to 0.7cm{\rightarrowfill}}
 \limits^{#1}}}

\newcommand{\restrict}[2]{\ensuremath{\left. #1 \right|_{#2}}}
\newcommand{\conv}{{\rm conv}}

\newcommand{\ctext}[1]{\raise0.2ex\hbox{\textcircled{\scriptsize{#1}}}}

\DeclareMathOperator{\Hom}{Hom}

\DeclareMathOperator*{\smallcup}{\reduce{\bigcup}}

\def\C{\mathbb C}
\def\cone{\mathrm{Cone\,}}

\def\D{\Delta}
\def\ep{\varepsilon}

\def\N{\mathbb N}

\def\p{\partial}
\def\P{\mathbb P}
\def\Q{{\mathbb Q}}
\def\R{{\mathbb R}}

\def\Vol{\mathrm{Vol}}

\def\Z{\mathbb Z}

\begin{document}

\title[Bott manifolds with vanishing Futaki invariants for all K\"ahler classes]{Bott manifolds with the strong Calabi dream structure}

\author{Hajime Ono}
\address{Department of Mathematics,
University of Tsukuba, 
1-1-1 Tennodai, Tsukuba, Ibaraki 305-8571 Japan}
\email{h-ono@math.tsukuba.ac.jp}

\author{Yuji Sano}
\address{Department of Applied Mathematics, 
Faculty of Science, 
Fukuoka University,
8-19-1 Nanakuma, Jonan-ku, 
Fukuoka 814-0180 Japan}
\email{sanoyuji@fukuoka-u.ac.jp}

\author{Naoto Yotsutani}

\address{Faculty of Education, Kagawa University,
1-1 Saiwai-cho, Takamatsu 760-8521, Japan}
\email{yotsutani.naoto@kagawa-u.ac.jp}

\makeatletter
\@namedef{subjclassname@2020}{%
  \textup{2020} Mathematics Subject Classification}
\makeatother


 \dedicatory{}
\date{\today}
\maketitle

\noindent{\bfseries Abstract.}
We prove that the only Bott manifolds such that the Futaki invariant vanishes for any K\"ahler class are isomorphic to the products of the projective lines.

\section{Introduction}
One of the main problems in K\"ahler geometry is the existence problem of {\emph{canonical}} K\"ahler metrics on a K\"ahler manifold with fixed K\"ahler class. 
This problem began with the case of K\"ahler-Einstein metrics as a Calabi conjecture which states that any real $(1,1)$-form representing the first Chern class is given by the Ricci form of some 
K\"ahler metric in a given K\"ahler class. This problem was solved by Aubin \cite{Aub76} and Yau \cite{Ya78} for the case of negative first Chern class, and by Yau for zero first Chern class. 
For the positive first Chern class case, several obstructions were found. One of them is the vanishing of the {\emph{Futaki invariant}} \cite{Fu83a}. 
This obstruction was generalized to the case of constant scalar curvature K\"ahler metrics (cscK metrics) by Futaki \cite{Fu83b}, Calabi \cite{Ca85}. 
On the other hand, if there exists a cscK metric in the fixed K\"ahler class, then it must be unique up to automorphisms \cite{BB17, CT08}.

As a generalization of cscK metrics, Calabi introduced {\emph{extremal}} K\"ahler metrics,
and constructed an explicit extremal K\"ahler metric in each K\"ahler class on the Hirzebruch surfaces \cite{Ca82}.
This led to the speculation that extremal K\"ahler metrics exist in every K\"ahler class for any compact K\"ahler manifold.
However, his expectation was not true in general \cite{Le85, WZ12}. 
In honor of Calabi's original vision, X. Chen and J. Cheng introduced the notion of a {\emph{Calabi dream manifold (CDM)}}, that is a compact K\"ahler manifold which admits an extremal K\"ahler metric in each K\"ahler class \cite[Definition $1.12$]{CC21}. 
At this moment, several examples of CDMs have been discovered, for instance, \cite{Ca82,Gu95, Hw94, ACGT-F08, SW13}. 
In particular, Apostolov-Calderbank-Gauduchon-T\o nnesen-Friedman established the {\emph{admissible construction}}, and then found CDMs among some classes of projective bundles. 
In \cite{BCT-F19}, Boyer-Calderbank-T\o nnesen-Friedman applied the admissible construction to {\emph{Bott manifolds}}, which is the main focus of this paper.
Also, they found many explicit examples of Bott manifolds with extremal K\"ahler metrics.

The existence of cscK/extremal metrics on a polarized manifold has been expected to be equivalent to K-polystability/relative K-polystability (Yau-Tian-Donaldson conjecture). 
In particular, the existence problem of extremal K\"ahler metrics on Bott Fano manifolds is relevant to 
classify certain class of toric Fano manifolds in terms of relative K-polystability
(see, Section $3.5$ in \cite{NSY23}). 
Specifically, the third author and B. Zhou partially classified relative K-polystable toric Fano threefolds using a combinatorial criteria of relative K-stability established in \cite{YZ19}.
Although it still remains an open problem to verify relative K-polystability for some of toric Fano threefolds (see, \cite[Table1]{YZ23}),
we can adapt Bott structure to verify relative K-polystability for certain toric Fano manifolds.
For example, toric Fano fourfold $\P_{(\P^1)^3}(\mathcal O\oplus \mathcal O(1,1,1))$, (which is labeled $L_1$ in \cite{Bat98}), is an example of Fano Bott manifolds at stage four, and the criterion in \cite{YZ19} cannot be applied. However, it is known that $L_1$ is a CDM, so that we conclude that $L_1$ is relatively K-polystable with respect to the anticanonical polarization (see, Corollary $1.9$ in \cite{NSY23}).

As a special class of CDM, Martinez-Garcia \cite{MG21} called a compact K\"ahler manifold $X$ a {\emph{strong Calabi dream manifold}} (SCDM) if $X$ admits a cscK metric in each K\"ahler class. 
It was proved in \cite[Corollary $1.7$]{MG21} that any smooth projective rational strong Calabi dream surface is either $\P^2$ or $\P^1\times \P^1$.
The main interest of this paper is to classify all SCDMs of arbitrary dimension. Our main result is the following.
\begin{theorem}\label{thm:main}
Let $X$ be a stage $n$ Bott manifold. Then the Futaki invariant of $X$ vanishes for any K\"ahler class if and only if
 $X$ is isomorphic to an $n$-fold product of the projective line
i.e., $X\cong (\P^1)^n$.
\end{theorem}

The uniqueness of cscK metrics in the fixed K\"ahler class implies the following splitting theorem on Bott manifolds with cscK metrics.
\begin{corollary}\label{cor:main}
Let $X$ be a stage $n$ Bott manifold which admits a cscK metric in each K\"ahler class.
Then $X$ is isomorphic to the products of the projective lines. Moreover, such a cscK metric in every K\"ahler class is the Riemannian product of $c_i\omega_{\mathrm{FS}}$ for $1\leqslant i \leqslant n$, 
where $c_i$ are some positive constants and $\omega_{\mathrm{FS}}$ is the Fubini-Study metric of the projective line.
\end{corollary}
Finally, we explain a related work to Theorem $\ref{thm:main}$. Following the Yau-Tian-Donaldson conjecture, we can consider an algebraic counterpart of SCDM: 
a polarized manifold $(X,L)$ which is K-polystable for any polarization. By using the notion of slope semistability in \cite{RT07}, which is a weaker notion of K-semistability, Fujita and the third author proved the following.
\begin{theorem}[\cite{FY24}]\label{thm:FY}
Let $X$ be a stage $n$ Bott manifold. For any ample $\Q$-line bundle $L$ on $X$, if $(X,L)$ is slope semistable along any smooth prime divisor $D$ in $X$, then $X\cong (\P^1)^n$.
\end{theorem}
As a corollary of Theorem $\ref{thm:FY}$, they concluded {that any strong Calabi dream Bott manifold is isomorphic to the products of projective lines. 
We remark that on a Bott manifold, the assumptions in Theorem $\ref{thm:main}$ and Theorem $\ref{thm:FY}$ would be equivalent under certain condition.
In fact, for any torus-invariant prime divisor $D$ on a Bott manifold $X$, we assume that $\bigl((L-\varepsilon D)^{\cdot n}\bigr)=0$, where $\varepsilon$ is the Seshadari constant of $D$ with respect to 
a polarization $L$. 
Under this assumption, we see that the Futaki invariant for the holomorphic vector field vanishing on $D$ is equal to the limit of the Donaldson-Futaki invariant of the degeneration to the normal cone of $D$, 
when the parameter goes to the Seshadri constant. Thus, we conclude that the vanishing of the Futaki invariant is equivalent to slope semistability for any smooth prime divisor on a Bott manifold with a given polarization, if $\bigl((L-\varepsilon D)^{\cdot n}\bigr)=0$. For the reader's convenience, we give more details in Section $\ref{sec:Appendix}$.
Although the conclusions are close to each other, the approaches of this article and \cite{FY24} are different: in the former, we go through 
integration over the moment polytope of $X$. 
In particular, a key ingredient of our proof is the classical {\emph{Brunn-Minkowski inequality}} for convex polytopes (see, Section $\ref{sec:BM_Ineq}$). 
In the latter, the proof was based on the computations of intersection numbers on Bott manifolds.

This paper is organized as follows. Section $\ref{sec:Toric}$ gives a quick review of toric varieties and their associated convex polytopes. Section $\ref{sec:Bott}$ collects basic facts of Bott manifolds.
In Section $\ref{sec:BottPC}$, we recall a combinatorial characterization of Bott manifolds in terms of their associated fan (see, Proposition $\ref{prop:BottPC}$).
We prove Theorem $\ref{thm:main}$ in Section $\ref{sec:Proof}$ by induction of the dimension of the Bott tower.
In the final section, we give the detailed explanation of the relation between Theorem $\ref{thm:main}$ and Theorem $\ref{thm:FY}$,
and clarify the difference between the main theorems of this article and the paper by Fujita and the third author \cite{FY24}.

Through the paper, we work over a complex number field $\C$, and a {\emph{manifold}} is a smooth irreducible complex projective variety. 

\vskip 7pt

\noindent{\bfseries Acknowledgements.}
This work was partially supported by the first author's JSPS KAKENHI Grant Number JP$17$K$05218$, JP$22$K$03281$,
second author's JP$22$K$03325$ and third author's JP$18$K$13406$, JP$22$K03316.
The second author is also supported by research funds from Fukuoka University (Grant Number $225001$-$000$).

\section{Toric geometry and convex polytopes}\label{sec:Toric}
This section collects some fundamental results in toric geometry which will be used at a later stage. 
See \cite{CLS11} and \cite{Ful93} for more details.  
A {\emph{toric variety}} $X$ is an algebraic normal variety with an effective holomorphic action of $T_\C:=(\C^\times)^n$ with $\dim_\C X=n$.
Let $T_\R:=(S^1)^n$ be the real torus in $T_\C$ and $\mathfrak{t}_\R$ be the associated Lie algebra. Let $N_\R:=J\mathfrak{t}_\R\cong \R^n$ be the associated Euclidean space where $J$ is the complex structure
of $T_\C$. We denote the dual space $\Hom (N_\R, \R)\cong \R^n$ of $N_\R$ by $M_\R$.
Setting the group of algebraic characters of $T_\C$ by $M$, we see that $M_\R=M\otimes_\Z\R$. 

Let us denote the dual lattice of $M$ by $N$. Then, $N$ consists of the algebraic one parameter subgroups of $T_\C$.
A {\emph{fan}} $\Sigma $ in $N_\R$ is defined to be a collection of 
strongly convex polyhedral lattice cones  
such that any face of a cone in $\Sigma$ is also an element of $\Sigma$ and the intersection of two cones in $\Sigma$ is a face of each. 
The {\emph{support}} of $\Sigma$ is defined as the union of all cones contained in 
$\Sigma$. We denote the support of $\Sigma$ by $|\Sigma|$, that is, $|\Sigma|=\smallcup_{\sigma\in \Sigma}\sigma$. A fan $\Sigma$ is said to be {\emph{smooth}} 
if any cone $\sigma\in \Sigma$ is smooth (i.e., its minimal generators form part of a $\Z$-basis of $N$). A fan $\Sigma$ is called {\emph{complete}} if its support is the whole space $N_\R$.
A fan $\Sigma$ is {\emph{polytopal}} if there exists a convex lattice polytope $P$ in $M_\R$ such that $0\in P$ and whose normal fan $\Sigma_P$ coincides with $\Sigma$.

From now on $X_\Sigma$ denotes the toric variety associated with a fan $\Sigma$ in $N_\R$. Then, geometric properties of $X_\Sigma$ can be characterized by properties of $\Sigma$ as follows.
\begin{proposition}[Theorem $3.1.19$ in \cite{CLS11}]
Let $X_\Sigma$ be the toric variety defined by a fan $\Sigma$ in $N_\R$. Then:
\begin{enumerate}
\item $X_\Sigma$ is a smooth variety if and only if $\Sigma$ is smooth.
\item $X_\Sigma$ is compact (in the Euclidean topology) if and only if $\Sigma$ is complete.
\item $X_\Sigma$ is projective if and only if $\Sigma$ is polytopal.
\end{enumerate}
\end{proposition}

Let $\Sigma$ be a complete fan in $N_\R$. For $1\leqslant k \leqslant n$, we let $\Sigma(k)$ denote the set of all $k$-dimensional cones in $\Sigma;$ $\Sigma(k):=\set{\sigma \in \Sigma |\dim \sigma=k}$.
For any finite subset $F\subset N_{\R}$, we set
\[
\cone(F):=\sum_{x\in F}\R_{\geqslant 0}\,x.
\]
Then, we define
\[
G_\Sigma:=\Set{\bs v_\sigma \in N | \sigma \in \Sigma(1): \text{ $\bs v_\sigma$ is primitive and $\cone(\set{\bs v_\sigma})=\sigma$ }}.
\]

\begin{definition}\rm
A nonempty subset $R\subset G_\Sigma$ is a {\emph{primitive collection}} of $\Sigma$, if it satisfies the conditions
\begin{enumerate}
\item[(a)] $\cone(R)\notin \Sigma$, \quad  and \quad~  (b) $\cone(R\setminus\set{x})\in \Sigma$ for every $x\in R$.
\end{enumerate}
We denote by $\mathrm{PC}(\Sigma)$ the set of primitive collections of $\Sigma$.
\end{definition}

Let $\Sigma_1$ and $\Sigma_2$ be smooth complete fans in $N_\R$. $\Sigma_1$ and $\Sigma_2$ are said to be {\emph{combinatorially equivalent}} if there exists a bijective map
$\psi:G_{\Sigma_1}\to G_{\Sigma_2}$ such that for any nonempty subset $R\subset G_{\Sigma_1}$, we have $\cone(R)\in \Sigma_1$ if and only if $\cone(\psi(R))\in \Sigma_2$.
Furthermore, $\mathrm{PC}(\Sigma_1)$ and $\mathrm{PC}(\Sigma_2)$ are {\emph{isomorphic}} if there exists a bijective map $\phi: G_{\Sigma_1}\to G_{\Sigma_2}$ which induces a well-defined
bijective map
\[
\phi_*:\mathrm{PC}(\Sigma_1) \longrightarrow \mathrm{PC}(\Sigma_2), \qquad R\longmapsto \phi(R).
\]
Let $P$ be a full dimensional convex polytope in $M_\R$, and let $\mathcal V(P)$ be the set of all vertices of $P$.
Here and hereafter, we assume unless otherwise stated that all polytopes are full dimensional convex polytopes.
Two convex polytopes $P_1$ and $P_2$ are {\emph{combinatorially equivalent}} if there exists a bijective map $\varphi:\mathcal V(P_1)\to\mathcal V(P_2)$ which induces a well-defined map
\begin{equation*}
\xymatrix@R=-1ex{
\varphi_*: \set{\text{faces of $P_1$}}\ar[r]& \set{\text{faces of $P_2$}}\\
\quad \! \din &\hspace{-0.5cm}\din\\
 F=\conv\set{\bs v_i|\bs v_i\in \mathcal V(P_1)} \ar@{|->}[r]&F'=\conv\set{\varphi(\bs v_i)|\varphi(\bs v_i)\in \mathcal V(P_2)}}
\end{equation*}
that preserves dimensions, intersections, and the face relation. For example, there is a seven dimensional non-symmetric reflexive K\"ahler-Einstein moment polytope $P\subset M_\R$ with $64$ vertices used in
\cite{OSY12}. This seven dimensional polytope $P$ is combinatorially equivalent to $\D_3\times(\D_1)^4$, where $\D_n$ denotes the standard $n$-simplex.

Finally, we recall that an $n$-dimensional convex polytope $P$ in $ M_\R$ is {\emph{Delzant}} if it satisfies the following conditions:
\begin{enumerate}
\item[(i)] At each vertex $p$ of $P$, there are exactly $n$ edges $l_{p,1}, \ldots, l_{p,n}$ of $P$ emanating from $p$.
\item[(ii)] Each edge $l_{p,k}$ in (i) is of the form $p+tu_k$ $(t\geqslant 0)$ with some primitive vector $u_k\in M$.
\item[(iii)] ~ For each vertex $p$, the primitive vectors $u_1, \ldots, u_n$ corresponding to the edges $l_{p,1}, \ldots, l_{p,n}$ generate the lattice $M$ over $\Z$.
\end{enumerate}
From these definitions, we readily obtain the following proposition.
\begin{proposition}[Proposition $2.1.6$ in \cite{Bat98}, Proposition $1.3.4$ in \cite{Sat02}]
Let $P_1$ and $P_2$ be lattice Delzant polytopes in $M_\R$. Let $\Sigma_{P_1}$ and $\Sigma_{P_2}$ be the associated normal fans in $N_{\R}$.
Then, the followings are equivalent.
\begin{enumerate}
\item[(a)] $P_1$ and $P_2$ are combinatorially equivalent.
\item[(b)] $\Sigma_{P_1}$ and $\Sigma_{P_2}$ are combinatorially equivalent.
\item[(c)] $\mathrm{PC}(\Sigma_{P_1})$ and $\mathrm{PC}(\Sigma_{P_2})$ are isomorphic.
\end{enumerate}
\end{proposition}
For simplicity, we use $X_P$ to denote the projective toric variety associated with the normal fan $\Sigma_P$ of a Delzant polytope $P\subset M_\R$.

\section{Bott manifolds}\label{sec:Bott}
Historically, Bott and Samelson introduced a family of complex manifolds obtained as the total spaces of iterated $\P^1$-bundles over $\P^1$ \cite{BS58}. 
The notion of {\emph{Bott tower}} was first defined by Grossberg and Karshon in their study of complex geometric structure of Bott-Samelson varieties \cite{GK94}.
Each Bott tower is a (smooth) projective toric variety whose corresponding moment polytope $P\subset M_\R$ is combinatorially equivalent to a cube. 
In this section, we recall the toric description of Bott manifolds and see their basic properties. For further details, we refer the reader to Section $7.8$ in \cite{BP15} and \cite{CMS10}.

\subsection{Definition and basic properties}\label{sec:DefBott}
For a given $n\in \N$ with $n\geqslant 2$, we inductively construct complex manifolds $X_k$ for $k=1,2, \cdots , n$ as follows.
Let $X_0$ be a point $\mathrm{Spec\:} \C$, and let $X_1$ be the projective line $\P^1$.
For $k\geqslant 2$, we assume that $X_{k-1}$ is already determined and choose a holomorphic line bundle $N_k$ on $X_{k-1}$.
Then, we define $X_k$ as the total space of the $\P^1$-bundle over $X_{k-1}$ and denote it by
\begin{equation}\label{eq:stagek}
X_k:=\P_{X_{k-1}}(\mathcal O_{X_{k-1}}\oplus N_k)\stackrel{\pi_k}{\longrightarrow} X_{k-1}.
\end{equation}
\begin{definition}\rm
A {\emph{Bott tower}} of height $n$ is a sequence of $\P^1$-bundles 
\begin{equation}\label{eq:BottTower}
X_n \stackrel{\pi_n}{\longrightarrow} X_{n-1} \stackrel{\pi_{n-1}}{\longrightarrow} \cdots \stackrel{\pi_2}{\longrightarrow} X_1 \stackrel{\pi_1}{\longrightarrow} \mathrm{Spec\:} \C
\end{equation}
over complex manifolds, where $X_1=\P^1$ and $X_k$ for $2\leqslant k \leqslant n$ is the 
$ \P^1$-bundle defined in $\eqref{eq:stagek}$.

The last stage $X_n$ in $\eqref{eq:BottTower}$ is referred to as a stage $n$ {\emph{Bott manifold}} although it is also often called as the same name {\emph{Bott Tower}}.
\end{definition}
A Bott tower $X_n$ is {\emph{topologically trivial}} if each projection $\pi_k:X_k\to X_{k-1}$ gives a trivial $\P^1$-bundle.
In particular, such $X_n$ is diffeomorphic to the product of real two dimensional spheres $(S^2)^n$.
The (topological) {\emph{twist}} is the number of topologically nontrivial $\P^1$-bundles in the Bott tower $\eqref{eq:BottTower}$ that is well-defined due to the work of Choi and Suh \cite{CS11}.
In fact, we may have several Bott tower structures for a given Bott manifold (cf. Section $3.2$ in \cite{FY24}).
Hence it is not obvious whether the number of nontrivial $\P^1$-fibration $\pi_k: X_k\to X_{k-1}$ in the sequence is well-defined or not.
The well-definedness of the twist of a Bott manifold was proved in \cite{CS11}, Corollary $3.3$.

Let $(\C^2_*)^n$ be $n$ copies of $\C^2_*:=\C^2\setminus\set{0}$.
Then, one can construct a stage $n$ Bott manifold $X_n$ by taking a quotient of $(\C^2_*)^n$
by an algebraic torus $T_\C$. Alternatively $X_n$ is obtained as the symplectic quotient of $(S^3)^n$ with respect to the free $(S^1)^n$-action as we shall see below.
Remark that the former gives a complex geometric viewpoint of Bott manifolds, while the latter gives a symplectic geometric viewpoint of them.

Let $A$ be a lower triangular unipotent matrix of the form
\begin{equation}\label{eq:matrix}
A=\begin{pmatrix}
1      &   0      &  \cdots   & 0         &  0       \\
  a_2^1         & 1 &  \cdots    &  0     &      0   \\
     \vdots      &  \vdots      & \ddots &  \vdots      &    \vdots   \\
   a_{n-1}^1  &    a_{n-1}^2     &  \cdots      & 1 &   0     \\
      a_n^1      &   a_n^2     &  \cdots     &   a_n^{n-1}    &1
\end{pmatrix}
\end{equation}
with $a_j^i\in \Z$. The induced action $\lambda$ of $\bs t=(t_1,\ldots, t_n)\in T_\C$ on 
\[
(\bs z; \bs w)=((z_1, \ldots ,z_n);(w_1, \ldots ,w_n))\in (\C^2_*)^n
\]
is given by
\begin{equation}
\begin{split}\label{eq:action}
\xymatrix@R=-1ex{
\lambda:  T_\C\times (\C^2_*)^n  \ar[r]&(\C^2_*)^n\\
\din &\din\\
(\bs t, (\bs z; \bs w))\ar@{|->}[r]& \left( (t_1z_1, \ldots , t_nz_n);
\left( \left( \prod_{i=1}^nt_i^{a_1^i}\right) w_1, \ldots, \left( \prod_{i=1}^nt_i^{a_n^i}\right) w_n
 \right)
\right).
}
\end{split}
\end{equation}
For the unit sphere $S^3$ in $\C^2_*$, there is the induced action of $(\R_{>0})^n$ on $(\C^2_*)^n$ by $\eqref{eq:action}$,
which is transverse to $(S^3)^n$. Hence the orbits of this action are bijectively correspond to orbits of the induced free $(S^1)^n$-action on $(S^3)^n$.
Consequently, the geometric quotient $(\C^2_*)^n/T_\C$ gives a complex $n$-dimensional compact manifold which is denoted by $X_n(A)$ in some literature.
In this way, an $n$-dimensional Bott manifold can be constructed as a quotient of $(\C^2_*)^n$ by an algebraic torus.
In particular, we see that any Bott manifold is a smooth projective toric variety.

\subsection{A Bott manifold and its associated fan}\label{sec:BottPC}
For a given stage $n$ Bott manifold $X$, let $\Sigma$ be its associated polytopal fan. 
Recall that an {\emph{$n$-cross polytope}} is a smooth polytope which is dual to an $n$-cube.
According to Corollary $3.5$ in \cite{MP08}, $\Sigma$ is combinatorially equivalent to the fan consisting of the cones over the faces of an $n$-cross polytope.
This means
\begin{equation}\label{eq:PC}
{\mathrm{PC}}(\Sigma)=\set{\set{\bs u_1, \bs e_1},\set{\bs u_2, \bs e_2}, \cdots , \set{\bs u_{n}, \bs e_{n}}}.
\end{equation}
Furthermore, for the Bott tower with a matrix $A$ given in $\eqref{eq:matrix}$, we may assume that $\bs e_1, \ldots , \bs e_n$ are the standard basis of $\R^n$ and each $\bs u_i$ is written as
\begin{equation}\label{eq:PN}
\bs u_i=-\sum_{j=1}^na^i_j\bs e_j
\end{equation}
for $i=1, \ldots , n$. For later use, we denote $\set{\bs u_i, \bs e_i}$ as $\set{\bs v_{2i-1}, \bs v_{2i}}$ in $\eqref{eq:PC}.$
In particular, if a stage $n$ Bott manifold $X$ has twist zero, then we have $\bs v_{2i-1}+\bs v_{2i}= 0$ for all $1\leqslant i \leqslant n$.
As we have seen in $\eqref{eq:PN}$, each column vector in the matrix $\eqref{eq:matrix}$ corresponds to the primitive generator $\bs v_{2i-1}$ (i.e., $\bs u_i$) for $1\leqslant i \leqslant n$.
Summing up these arguments, we have the following proposition.
\begin{proposition}\label{prop:BottPC}
Let $X$ be a stage $n$ Bott manifold having twist $k$ with the associated fan $\Sigma \subset N_\R$.
Then, the primitive collection of $\Sigma$ has the form of
\[
{\mathrm{PC}}(\Sigma)=\set{\set{\bs v_1, \bs v_2},\set{\bs v_3, \bs v_4}, \cdots , \set{\bs v_{2n-1}, \bs v_{2n}}}, \qquad \rho_i=\cone(\bs v_i)\in \Sigma(1)
\]
for $1\leqslant i \leqslant 2n$. In particular, there exist integers 
\[
i_1<i_2< i_3<\cdots < i_{n-k} \qquad \text{with} \qquad\set{i_1, i_2, i_3, \cdots, i_{n-k}}\subset \set{1, 3,5,\cdots ,2n-1}
\]
satisfying $\bs v_{i_j}+\bs v_{i_j+1}=0$.
\end{proposition}

\section{The proof of Theorem $\ref{thm:main}$}\label{sec:Proof}
In this section, we prove the main result. Our proof consists of the following steps:
\begin{step}
For a stage $n$ Bott manifold $X$, we fix a K\"ahler class $\Omega$ of $X$.
Denoting by $P$ the corresponding moment polytope of $(X,\Omega)$, we see that $P$ is written in the form of $\eqref{eq:Polytope Pinf}$.
Then, the Futaki invariant of our targeted Bott manifold $(X,\Omega)$ is given by $\eqref{eq:DFheight}$.  
\item We prove that if the Futaki invariant of the vector field generated by $v=-x_n\p/\p x_n$ is identically zero,
where $x_n$ corresponds to the ``height" of $P$, then the volume of the slice
\[
Q_x=P\cap \set{x_n=x}
\]
is constant (Proposition \ref{prop:ConstFun}).
\end{step}
\begin{step}
We prove that if the Futaki invariant of the vector field generated by $v=-x_n\p/\p x_n$ is identically zero,
where $x_n$ corresponds to the ``height" of $P$, then the volume of the slice
\[
Q_x=P\cap \set{x_n=x}
\]
is constant (Proposition \ref{prop:ConstFun}).
\end{step}
\begin{step}
In a more general setting, looking at the projection $X\to \P^1$ for a smooth projective variety
and considering the lift of $\C^{\times}$-action on $\P^1$ to $X$, we consider the vector field $v$ generated by this $\C^{\times}$-action.
We can show that if the Futaki invariant of $v$ vanishes on the entire space of K\"ahler cone, then $X$ splits into $\P^1$ and another smooth projective variety $X'$. 
In the toric setting, using the Brunn-Minkowski inequality for convex polytopes, we prove it in Corollary \ref{cor:splitting}.
\end{step}
\begin{step}
For the splitting $X_P=X_Q\times \P^1$ in {\it{Step}} $3$, we see that the Futaki invariant of $X_P$ is proportional to the Futaki
invariant of the base manifold $X_Q$ in Lemma \ref{lem:Futaki}. By an inductive argument in the dimension of a Bott manifold,
we verify the assertion. 
\end{step}
Following the outline described in the above, we shall give a proof as follows.
Let $X$ be the stage $n$ Bott manifold associated with a smooth polytopal fan $\Sigma$.
Then, $\Sigma(1)$ consists of $2n$ rays such as
\[
\Sigma(1)=\set{\rho_1, \rho_2, \ldots, \rho_{2n-1}, \rho_{2n}},\,\rho_i={\mathrm{Cone\,}} (\bs v_{i}), 
\]
where $\bs v_{i}$ is the primitive generator of each ray $\rho_i$.
Since a Bott tower $X$ (of height $n$) has the trivial $\P^1$-bundle at the first stage, i.e., $X_1=\P^1$, we may assume
\begin{equation}\label{eq:P1_line}
	\bs v_{2n-1} + \bs v_{2n} =\bs 0, \qquad  \text{with} \quad \bs v_{2n}=\bs e_n.
\end{equation}
Let $D_i$ be the $T_\C$-invariant Cartier divisor on $X$ corresponding to $\rho_i\in \Sigma(1)$.
We consider a K\"ahler class $\Omega\in H^{1,1}(X,\R)$ given by
\[
\Omega=\sum_{i=1}^{2n}a_i[D_i], \quad a_i\in \R,
\]
where $[D_i]$ denotes the Poincar\'e dual of $D_i$.
Then, the corresponding moment polytope of  
$(X,\Omega)$ is (a part of) the pillar given by
\begin{equation}\label{eq:polytope_P}
	P=\set{\bs x \in M_\R| \braket{\bs x, \bs v_{i} }\geqslant -a_{i}, \, i=1, \ldots ,2n}
\end{equation}
with $\bs x=(x_1, \ldots ,x_n)$.

\subsection{Volume preserving constraints for each slice of $P$}\label{sec:DeriFut}
Through this subsection, we fix $a_i$ for each $1 \leqslant i \leqslant 2n-2$ in $\eqref{eq:polytope_P}$.
Let $P_{(-\infty, +\infty)}$ be the unbounded pillar defined by
\[
		P_{(-\infty, \, +\infty)}=\set{\bs x \in M_\R| \braket{\bs x, \bs v_{i} }\geqslant -a_{i}, \, i=1, \ldots ,2n-2}.
\]
We regard $a_{2n-1}$ and $a_{2n}$ as real parameters $s$ and $-t$ respectively.

Then, the polytope $P$ can be written as 
\begin{equation}\label{eq:Polytope Pinf}
	P_{[t, \, s]}
	=
	P_{(-\infty, \, +\infty)} 
	\cap 
	\{ t\leqslant x_{n} \leqslant s\}.
\end{equation}
Let us denote the slice of $P_{(-\infty,\, +\infty)}$ at $x_n=x$ by $Q_x$:
\begin{equation}\label{eq:PolytopeQ_x}
	Q_{x} = P_{(-\infty,\, +\infty)}\cap \{x_{n}=x\}.
\end{equation}
Taking a suitable parallel translation of $P_{(-\infty, +\infty)}$ along the $x_n$-axis, we may assume $t=0$ in $\eqref{eq:Polytope Pinf}$. Henceforth we shall take $t=0$ unless otherwise stated.

Let $\ell_i(\bs x)=\braket{\bs x, \bs v_i}+a_i$ be the defining affine function of the facet $F_i$ of $P$.
Let $dv=dx_1\wedge \cdots \wedge dx_n$ be the standard volume form of $M_\R$.
On each facet $F_i=\set{\bs x\in P| \ell_i(\bs x)=0}\subset \partial P$, we define the $(n-1)$-dimensional Lebesgue measure $d\sigma_i$ by
\begin{equation}\label{def:Bound_Mes}
dv=\pm d\sigma_i \wedge d\ell_i.
\end{equation}
Analogously, we consider the relative volume $\Vol (\partial Q_x)=\int_{\partial Q_x}d\sigma_{Q_x}$, where $d\sigma_{Q_x}$ is the $(n-2)$-dimensional Lebesgue measure 
defined as follows: let $d\sigma_i$ be the volume form of the facet $F_i$ defined in $\eqref{def:Bound_Mes}$. 
Then, the corresponding $(n-2)$-dimensional Lebesgue measure $d\sigma_{Q_x,i}$ on 
$\partial Q_x\cap F_i$ is defined to be 
\begin{equation}\label{def:Bound_Mes2}
d\sigma_{Q_x,i} \wedge dx_n=\pm d\sigma_i.
\end{equation} 
For a fixed $s>0$, let us consider a convex polytope given by
\begin{equation}\label{eq:SlicePoly}
P_{[0,s]}=P_{(-\infty,+\infty)}\cap\set{0\leqslant x_n \leqslant s},
\end{equation}
which has a pair of parallel facets $Q_s$ (the top facet) and $Q_0$ (the bottom facet).
Here and hereafter, we shall abbreviate the polytope $P_{[0,s]}$ by $P$.
The following lemma is an immediate consequence of the definition of $d\sigma_{Q_x,i}$ in $\eqref{def:Bound_Mes2}$. 
\begin{lemma}\label{lem:boundary}
Let $s$ be a fixed positive constant.
For any $0\leqslant u \leqslant s$, we consider 
the slice $Q_u$ of $P$ given in $\eqref{eq:PolytopeQ_x}$, so that each $Q_u$ is parallel to both $Q_0$ and $Q_s$.
Regarding $Q_u$ as a full-dimensional polytope in $\R^{n-1}$, we denote the standard volume form of $Q_u$ by $dv_{Q_u}$.
Let us define the functions $f(u)$ and $g(u)$ by
\[
f(u):=\Vol(Q_u)=\int_{Q_u}dv_{Q_u} \qquad \text{and} \qquad
g(u):=\Vol(\partial Q_u)=\int_{\partial Q_u}d\sigma_{Q_u}
\]
respectively. Then, the volume of $\partial P$, and the barycenter of $\partial P$ with respect to the $x_n$-axis are given as follows:
\begin{equation}\label{eq:boundary}
\begin{aligned}
\int_{\partial P} d\sigma&=\int_0^s g(u)\, du+f(0)+f(s), &\text{and} \\
\int_{\partial P} x_n \, d\sigma&=\int_0^s ug(u)\, du+sf(s).
\end{aligned}
\end{equation}
\end{lemma}
\begin{proof}
Firstly, we remark that the polytope $P=P_{[0,s]}$ is combinatorially equivalent to the product of $Q_{0}$ and $[0,s]$.
Thus, the boundary $\partial P$ of $P$ is also combinatorially equivalent to
\begin{equation}\label{eq:decompos}
\partial Q_{0}\times [0,s]\cup Q_{0}\times\set{0,s}.
\end{equation}
This implies that
\begin{align*}
\Vol(\partial P)&=\int_{\partial P}d\sigma
=\int_0^s\left(\int_{\partial Q_u}1\,d\sigma_{Q_u}\right)du+\restrict{\int_{Q_u}1\,dv_{Q_u}}{u=0}+\restrict{\int_{Q_u}1\,dv_{Q_u}}{u=s}\\
&=\int_0^s\Vol(\partial Q_u)\,du+\Vol(Q_0)+\Vol(Q_s)\\
&=\int_0^sg(u)\,du+f(0)+f(s).
\end{align*}
This proves the first equality of $\eqref{eq:boundary}$.

For the second equality of $\eqref{eq:boundary}$, we compute
\begin{align*}
\int_{\partial P}x_n\, d\sigma&=\int_0^su\left(\int_{\partial Q_u}1\,d\sigma_{Q_u} \right)du
+u\restrict{\int_{Q_u}1\,dv_{Q_u}}{u=0}+u\restrict{\int_{Q_u}1\,dv_{Q_u}}{u=s}\\
&=\int_0^sug(u)\,du+sf(s).
\end{align*}
Thus, the assertions are verified.
\end{proof} 
Using the terminology defined in $\eqref{eq:SlicePoly}$, we recall that the {\emph{Futaki invariant}} of the vector field generated by  $v=-x_{n}\p /\p x_{n}$ on
the toric K\"ahler manifold $\bigl(X_{P_{[0,s]}}, \Omega_{P_{[0,s]}} \bigr)$ is equal to
\begin{equation}\label{eq:DFheight}
F_{X_{P_{[0, \, s]}}}(v)
=\Vol (P_{[0, \, s]})\int_{\p P_{[0, \, s]}} x_{n}\, d\sigma -\Vol(\p P_{[0, \, s]})\int_{P_{[0, \, s]}} x_{n}\, dv.
\end{equation}
See, for example \cite{D02}.
\begin{proposition}\label{prop:ConstFun}
For a fixed $s>0$ with $\dim Q_0=\dim Q_s=n-1$, we assume that $F_{X_{P_{[0, \, x]}}}(v)=0$ for any $x\in[0, \, s]$, as long as $Q_x$ has dimension $n-1$.
Then, the volume $\Vol(Q_{x})$ is constant for any $x\in[0,s]$. 
\end{proposition}
\begin{proof}	
We use the same notation as in Lemma $\ref{lem:boundary}$. Using the equalities in $\eqref{eq:boundary}$, we set 
\begin{equation}\label{def:PolyFun}
\begin{aligned}
		a(x)
	&:=
		\Vol(P_{[0, \, x]})= \int^x_0 f(u)du,
	\\
		b(x)
	&:=
		\Vol(\p P_{[0, \, x]})
		=
		\int^x_0 g(u)du + f(0) + f(x),
	 \\
	 	c(x)
	 &:=
	 	\int_{P_{[0, \, x]}} x_n dv
	 	=
	 	\int^x_0 u f(u) du, \quad \text{and}
	 \\
	 	d(x)
	 &:=
	 	\int_{\p P_{[0, \, x]}} x_n d\sigma
	 	=
	 	\int^x_0 u g(u) du + xf(x).
\end{aligned}
\end{equation}
Then, one can see that $a(x), b(x), c(x)$ and $d(x)$ are all polynomial functions. In particular,
we have
\[
	F_{X_{P_{[0, \, x]}}}(v) = a(x) d(x) - b(x)c(x).
\]
\begin{lemma}\label{lem:abcd}
Let $a^{(m)}(x)$ (resp. $b^{(m)}(x), c^{(m)}(x), d^{(m)}(x)$) denote the $m$-th derivative of the function $a(x)$ (resp. $b(x), c(x), d(x)$). Then, we have the following:
\[
\begin{array}{lclcl}
			a(0)=0, &\quad &
		 a^{(m)}(0)= f^{(m-1)}(0), &\quad & \text{for}\quad m \geqslant 1,\\ \vspace{-0.35cm} \\
			b(0)=2f(0), &\quad &
		 b^{(m)}(0)= g^{(m-1)}(0)+ f^{(m)}(0), &\quad & \text{for}\quad  m\geqslant 1,\\ \vspace{-0.35cm} \\
			c(0)=c^{(1)}(0)=0, &\quad &
	       c^{(m)}(0)= (m-1)f^{(m-2)}(0), &\quad & \text{for}\quad m\geqslant 2, \\ \vspace{-0.35cm} \\ 
			d(0)=0, &\quad & 
			d^{(1)}(0)=f(0)=\Vol(Q_0), \\ \vspace{-0.35cm} \\ 
                        &\quad & d^{(m)}(0)= (m-1)g^{(m-2)}(0)+ mf^{(m-1)}(0), &\quad & \text{for}\quad m\geqslant 2.
\end{array}	
\]
\end{lemma}
\begin{proof}[Proof of Lemma $\ref{lem:abcd}$]
The values of $a(0), b(0), c(0)$ and $d(0)$ are obvious from $\eqref{def:PolyFun}$.

All other equalities follows from the direct computations using the Leibniz rule.
Firstly, we find that $a^{(m)}(x)=f^{(m-1)}(x)$ for a positive integer $m$ by calculating the derivatives of $a(x)$. This yields that
\begin{equation*}
a^{(m)}(0)= f^{(m-1)}(0).
\end{equation*}
Next we compute $b^{(m)}(0)$. The straightforward computation shows that
\begin{equation}\label{eq:b^m(s)}
b^{(m)}(x)=g^{(m-1)}(x)+f^{(m)}(x),
\end{equation}
for $m\geqslant 1$. Plugging $x=0$ into $\eqref{eq:b^m(s)}$, for $m\geqslant1$, we conclude that
\begin{equation*} 
b^{(m)}(0)=g^{(m-1)}(0)+f^{(m)}(0). 
\end{equation*}
Let us calculate the derivatives of $c(x)$. By Leibniz rule, we find that
\begin{align*}
c^{(m)}(x)=(xf(x))^{(m-1)}&=\sum_{i=0}^{m-1}{m-1 \choose i} x^{(i)}f^{(m-1-i)}(x) \\
&=xf^{(m-1)}(x)+(m-1)f^{(m-2)}(x).
\end{align*}
Thus, we have
\begin{equation*}
c^{(1)}(0)=0 \qquad \text{and} \qquad c^{(m)}(0)=(m-1)f^{(m-2)}(0) \quad  \text{for} \quad m\geqslant 2.
\end{equation*}
Finally, we compute the derivatives of $d(x)$. The direct computation shows that
\begin{align*}
d^{(m)}(x)&=(xg(x))^{(m-1)}+(xf(x))^{(m)}\\
&=\sum_{i=0}^{m-1}{m-1\choose i}x^{(i)}g^{(m-1-i)}(x)+\sum_{j=0}^{m}{m\choose j}x^{(j)}f^{(m-j)}(x)\\
&=xg^{(m-1)}(x)+(m-1)g^{(m-2)}(x)+xf^{(m)}(x)+mf^{(m-1)}(x).
\end{align*}
Consequently, we have
\begin{align*}
d^{(1)}(0)&=f(0)=\Vol(Q_0), \quad \text{and} \qquad \\
d^{(m)}(0)&=(m-1)g^{(m-2)}(0)+mf^{(m-1)}(0),
\end{align*}
for $m\geqslant 2$.
\end{proof}
\noindent
By Lemma \ref{lem:abcd}, we find 
\begin{align}\label{eq:3rdDeriv}
		\begin{split}
\frac{d^3}{d x^3}F_{X_{P_{[0, \, x]}}}(v)\bigg|_{x=0} 
	&=
		a^{(3)}(0)d(0)+3a^{(2)}(0)d^{(1)}(0)
		+3a^{(1)}(0)d^{(2)}(0)+a(0)d^{(3)}(0)
	\\
	&
		-b^{(3)}(0)c(0)-3b^{(2)}(0)c^{(1)}(0)
		-3b^{(1)}(0)c^{(2)}(0)-b(0)c^{(3)}(0)
	\\
	&=
		 2 f'(0)f(0).
		 \end{split}
\end{align}
Since $f(0)>0$ and $F_{X_{P_{[0, \, x]}}}(v)$ is identically $0$ by our assumption, we have $f'(0)=0$.
By taking an appropriate parallel translation of $P_{(-\infty,\, +\infty)}$ along the $x_n$-axis again if necessary, we conclude that $f'(x)$ is identically $0$ for arbitrary $x\in[0,\, s]$.
Thus, the volume $\Vol(Q_{x})$ is constant.
\end{proof}

\subsection{The Brunn-Minkowski inequality for convex polytopes}\label{sec:BM_Ineq}

In this subsection, we prove that any two slices $Q_x$ and $Q_{x'}$ for $0\leqslant x,x'\leqslant s$ are congruent. 
For this, we apply the Brunn-Minkowski inequality which estimates the (Euclidean) volume $\Vol(C+D)$ of the Minkowski sum
\[
C+D=\Set{\bs x+\bs y| \bs x \in C, ~ \bs y\in D }
\]
for two convex bodies $C$ and $D$. We refer the reader to \cite{Gr07} for more detail. 

A set $C$ in $\R^n$ is said to be {\emph{convex}} if it satisfies
\[
(1-\lambda)\bs x+\lambda \bs y \in C
\]
for any $\bs x, \bs y\in C$ and $0\leqslant \lambda \leqslant 1$. We call a compact convex set a {\emph{convex body}}.
A convex body $C$ is {\emph{proper}} if its interior is non-empty, and {\emph{improper}} otherwise.
Let $\mathcal C(\R^n)$ be the space of all convex bodies in $\R^n$. We say that two convex bodies $C,D \in \mathcal{C}(\R^n)$ are {\emph{homothetic}} to each other, 
if there exist $r>0$ and $\bs a\in \R^n$ such that
\[
C=rD+\bs a=\set{\bs x\in \R^n | \bs x=r\bs y+ \bs a, ~~  \bs y\in D}.
\]
Considering the Minkowski sum 
\[
(1-\lambda)C+\lambda D \qquad \text{in~~~}~~~ \R^n
\]
for $C, D\in \mathcal{C}(\R^n)$ and $0\leqslant \lambda \leqslant 1$, we introduce the function $\Psi$ defined by
\begin{equation}\label{def:Vol_fun}
\xymatrix@R=-1ex{
\Psi: [0,1] \ar[r]&\R\\
\qquad  \!\! \din &\hspace{0.0cm}\din\\
 \qquad\!\! \lambda  \ar@{|->}[r]&\Vol\bigl((1-\lambda)C+\lambda D \bigr)^{\frac{1}{n}}.}
\end{equation}
In order to prove Proposition $\ref{prop:congr}$ below, we need the following result. See Theorem $8.3$ in \cite{Gr07} for the proof.
\begin{theorem}[The Brunn-Minkowski inequality]\label{thm:BM_Ineq}
Let $\Psi$ be the function defined in $\eqref{def:Vol_fun}$. If $\Psi$ is \emph{not} strictly concave, then either one of the following conditions is satisfied:
\begin{enumerate}
\item[(i)] $C$ or $D$ is a point.
\item[(ii)] Both $C$ and $D$ are improper, and lie in parallel hyperplanes.
\item[(iii)] Both $C$ and $D $ are proper, and they are homothetic to each other.
\end{enumerate}
\end{theorem}
We state the key result for proving Theorem $\ref{thm:main}$.
\begin{proposition}\label{prop:congr}
Let us take any two distinct $u_1$ and $u_2$ in $[0, s]$ with $u_1<u_2$.
For any $x$ between $u_1$ and $u_2$, let $Q_x$ be the slice of $P_{(-\infty,\,+\infty)}$ defined in $\eqref{eq:PolytopeQ_x}$.
If $\Vol(Q_x)$ is constant for all $x\in [u_1, u_2]$,
then $Q_{u_1}$ and $Q_{u_2}$ are congruent.
\end{proposition}
\begin{proof}
Denoting $C:=Q_{u_1}$ and $D:=Q_{u_2}$, we remark that the Minkowski sum $(1-\lambda)C+\lambda D$ is the slice $Q_{(1-\lambda)u_1+\lambda u_2}$.
Let us define $\widetilde{Q}_\lambda:=(1-\lambda)C+\lambda D$.
Then, 
\[
\Vol(\widetilde{Q}_\lambda)=\int_{\widetilde{Q}_\lambda} dx_1\wedge \cdots \wedge dx_{n-1}
\]
is a constant function in $\lambda$ by assumption. 
Hence, the function
\[
\Psi:[0,1] \longrightarrow \R, \qquad \lambda \longmapsto \Vol(\widetilde{Q}_\lambda)^{\frac{1}{n-1}}
\]
is not strictly concave. 
If we regard each $\widetilde{Q}_\lambda$ as a convex polytope in the hyperplane $\set{x_n=\lambda}\cong\R^{n-1}$,
we note that
\begin{enumerate}
\item[(i)] neither $C$ nor $D$ is a point, and
\item[(ii)] both $C$ and $D$ are proper convex polytopes in $\R^{n-1}$, 
\end{enumerate}
because $\dim \widetilde{Q}_\lambda=n-1$ in $\R^{n-1}$ for any $\lambda$. 
By applying Theorem $\ref{thm:BM_Ineq}$
into $C$ and $D$ in $\R^{n-1}$, we conclude that $C$ and $D$ are homethetic, i.e., there exist some $r>0$ and $\bs a\in \R^{n-1}$ such that
\[
C=rD+\bs a \quad \Leftrightarrow \quad Q_{u_1}=rQ_{u_2}+\bs a.
\]
From this, we find 
\begin{equation}\label{eq:homothe}
\Vol(Q_{u_1})=r^{n-1}\Vol(Q_{u_2}).
\end{equation}
Since $\Vol(Q_{u_1})=\Vol(Q_{u_2})$, 
we conclude $r=1$. Consequently, $Q_{u_1}$ and $Q_{u_2}$ are congruent.
\end{proof}

\begin{corollary}\label{cor:splitting}
If the Futaki invariant $F_{X_{P_{[0, \, s]}}}(v)$ in $\eqref{eq:DFheight}$ vanishes for any $ s>0$, then our stage $n$ Bott manifold $X$ splits into $X' \times \P^{1}$, where 
$X'$ is the $(n-1)$-dimensional smooth toric variety whose associated fan $\Sigma'$ has 
the primitive collection
\begin{align*}
	{\mathrm{PC}}(\Sigma')&=\set{\set{\bs v_1, \bs v_2},\set{\bs v_3, \bs v_4}, \cdots , \set{\bs v_{2n-3}, \bs v_{2n-2}}}\\
	&=\Set{\set{\bs u_1, \bs e_1}, \set{\bs u_2, \bs e_2}, \cdots , \set{\bs u_{n-1}, \bs e_{n-1}}}.
\end{align*}
Moreover, the Futaki invariant on the base manifold $X'$ for any K\"ahler class also vanishes, whenever the Futaki invariants on $X$ for all K\"ahler classes vanish.
\end{corollary}
\begin{proof}
Let $P$ be the moment polytope of a toric K\"ahler manifold $(X_{P_{[0,s]}}, \Omega_{P_{[0,s]}})$.
By Proposition $\ref{prop:congr}$, the moment polytope $P$ is a quadrangular prism with the bottom
\[
Q=\set{\bs x\in M_\R| x_n=0,  \braket{\bs x, \bs v_i}\geqslant -a_i, ~ i=1, \ldots , 2n-2}.
\]
By taking an appropriate parallel translation of $P$ along the $x_n$-axis if necessary, we assume that $a_{2n-1}=0$ and $a_{2n}=a$ for some constant $a>0$. 
Since $P$ is a quadrangular prism with the bottom facet $Q$,
we have the splitting of $P$ such that
\begin{equation}\label{eq:split}
P=Q\times [0, a].
\end{equation}
The first statement follows from the fact that $\Sigma'$ is the normal fan of $Q$. 

The second assertion is a consequence of Lemma $\ref{lem:Futaki}$. 
We remark that the splitting $\eqref{eq:split}$ implies that the associated toric manifold $X_Q$ is the torus invariant toric divisor of $X_P$. Then, we have the following:
\begin{lemma}\label{lem:Futaki}
For each $i=1, \ldots, n$, let $F_{X_P}(x_i)$ be the Futaki invariant on $(X_P, \Omega_P)$ given by
\begin{equation}\label{eq:Futaki_X_P}
F_{X_P}(x_i):=\Vol(P)\int_{\partial P}x_i\,d\sigma-\Vol(\partial P)\int_P x_i\,dv.
\end{equation}
Then, $F_{X_P}(x_i)$ is proportional to the Futaki invariant $F_{X_Q}(x_i)$ on the base manifold $(X_Q,\Omega_Q)$, for each $i=1,\ldots, n-1$.
\end{lemma}
Hence, the second assertion of Corollary \ref{cor:splitting} is verified.
\end{proof}
\begin{proof}[Proof of Lemma $\ref{lem:Futaki}$]
Let $\bs x=(x_1,\ldots ,x_n)$ be the coordinate function of $M_\R$, and let $\widetilde {\bs x}=(x_1, \ldots, x_{n-1})$ be its projection. We note that the splitting $\eqref{eq:split}$ yields that the boundary of $P$ can be written as
\begin{equation}\label{eq:Bound_Split}
\partial P=\partial Q\times [0, a]\cup Q\times \Set{0, a}.
\end{equation}
We also remark that two measures on $\partial Q$, that is,
\begin{itemize}
\item the measure given by the restriction of $d\sigma$ (on $\partial P$) to $\partial Q$, and
\item the induced measure on $\partial Q$ when we regard $Q$ as a full dimensional polytope in the hyperplane $\set{x_n=c}$,
where $c$ is some constant,
\end{itemize}
coincide with each other. Thus, by $\eqref{eq:split}$ and $\eqref{eq:Bound_Split}$, we have
\[
\Vol(P)=a\Vol(Q) \qquad \text{and} \qquad \Vol(\partial P)=a\Vol(\partial Q)+2\Vol(Q).
\]
Furthermore, for $i=1, \ldots, n-1$, we find
\begin{align*}
\int_Px_i\,dv&=a\int_Q x_i\,dv_Q, \qquad  \hspace{2.2cm} \int_Px_n \,dv=\frac{a^2}{2}\Vol(Q),\\
\int_{\partial P}x_i\,d\sigma&=a\int_{\partial Q}x_i\,d\sigma_Q+2\int_Qx_i\,dv_Q, \quad \text{and}\\
\int_{\partial P}x_n\, d\sigma&=\frac{a^2}{2}\Vol(\partial Q)+a\Vol(Q)
\end{align*}
respectively. Plugging these values into $\eqref{eq:Futaki_X_P}$, we see that
\begin{align*}
F_P(x_i)&=a^2\left(\Vol(Q)\int_Q x_i\,d\sigma_Q -\Vol(\partial Q)\int_Qx_i\, dv_Q\right)\\
&=a^2F_Q(x_i)
\end{align*}
for $i=1,\ldots, n-1$. 
\end{proof}
\begin{remark}\label{rem:splitting}\rm
Moreover, we can see a more general result of Corollary \ref{cor:splitting} as follows.
Let $X\to Y$ be a fibration for smooth projective varieties $X$ and $Y$.
Let $L_X$ and $L_Y$ be very ample line bundles on $X$ and $Y$ respectively.
If a vector field $v$ on $Y$ lifts to $\tilde v$ on $X$, then the Futaki invariant $F_{L_X+jL_Y}(\tilde v)$ of $\tilde v$ with respect to the K\"ahler class $L_X+jL_Y$ for $j$ sufficiently large has the leading order term given by the Futaki invariant $F_{L_Y}(v)$ of $v$ with respect to $L_Y$. 
This yields that if the Futaki invariant of $\tilde v$ vanishes identically on the K\"ahler cone of $X$, then the Futaki invariant of $ v$ also vanishes identically on the K\"ahler cone of $Y$.
\end{remark}

\subsection{The induction on the dimension of an SCD Bott manifold}\label{sec:induction}
Now we complete the proof of Theorem $\ref{thm:main}$ by an inductive argument in $n$. 
Assume that $X$ is an $n$-dimensional Bott manifold with vanishing Futaki invariants for all K\"ahler classes.
By Corollary $\ref{cor:splitting}$, we see that $\bs u_1, \ldots , \bs u_{n-1}$ are perpendicular to $\bs e_n$, for primitive generators in $\eqref{eq:PC}$. Hence, we have
\begin{equation}\label{eq:Perp}
\braket{\bs u_i, \bs e_n}=0,  \qquad \text{for} \quad i=1, \ldots , n-1
\end{equation}
with $\bs u_i=-\sum_{j=1}^na_j^i \bs e_j$. Then, $\eqref{eq:Perp}$ is equivalent to
\[
a_n^1=a_n^2=\cdots =a_n^{n-1}=0
\]
which implies that the matrix $A$ given in $\eqref{eq:matrix}$ is written as
\begin{equation*}
A=\begin{pmatrix}
1      &   0      &  \cdots   & 0         &  0       \\
  a_2^1         & 1 &  \cdots    &  0     &      0   \\
     \vdots      &  \vdots      & \ddots &  \vdots      &    \vdots   \\
   a_{n-1}^1  &    a_{n-1}^2     &  \cdots      & 1 &   0     \\
      0      &   0     &  \cdots     &  0    &1
\end{pmatrix}.
\end{equation*}
Let us denote the $(n-1)$-square submatrix 
\begin{equation*}
\begin{pmatrix}
1      &   0      &  \cdots    & 0& 0               \\
  a_2^1         & 1 &  \cdots    & 0&  0        \\
     \vdots      &  \vdots      & \ddots &  \vdots &\vdots        \\
      a_{n-2}^1  &    a_{n-2}^2     &  \cdots   & 1   & 0 \\
   a_{n-1}^1  &    a_{n-1}^2     &  \cdots   & a_{n-1}^{n-2}   & 1      
\end{pmatrix}
\end{equation*}
by $\widetilde{A}$ which corresponds to the stage $n-1$ Bott manifold $X'$ in Corollary $\ref{cor:splitting}$. 
Inductively, we apply Corollary $\ref{cor:splitting}$ into $\widetilde A$ which yields
\[
a_{n-1}^1=a_{n-1}^2=\cdots =a_{n-1}^{n-2}=0.
\]
As a result, we conclude that $A$ is the unit matrix. Thus, Theorem $\ref{thm:main}$ has been proved.

\section{Appendix}\label{sec:Appendix}
In this section, we shall see that Theorem \ref{thm:main} and Theorem \ref{thm:FY} are equivalent as long as we work over a complex number field $\C$.

Firstly, we recall the argument in \cite{FY24} briefly.
Let $X$ be an $n$-dimensional smooth projective variety, and let $L$ be an ample $\Q$-line bundle on $X$.
For an irreducible smooth subvariety $Z$ in $X$ with $\mathrm{Codim}_X(Z)=c$, we take the blow-up $\sigma:\hat{X}\rightarrow X$ along $Z$.
Denoting the ideal sheaf of $Z$ by $I_Z$, we consider the Cartier divisor $F$ in $\hat{X}$ defined by $\mathcal{O}_{\hat{X}}(-F)=I_Z\cdot\mathcal{O}_{\hat{X}}$.
We denote the Seshadri constant of $Z$ with respect to $L$ by
\[
	\ep=\ep_L(Z)=\max\set{x\in\R_{>0}| \sigma^*L-xF 	\text{ is nef on } \hat X}.
\]
Setting
\[
\mu_L=\frac{\bigl(L^{\cdot n-1}\cdot(-K_X) \bigr)}{(L^{\cdot n})},
\]
we define the invariant $\xi_L(Z)$ by
\begin{equation}\label{def:xi}
\xi_L(Z)=n\int_0^\ep\bigl((\sigma^*L-xF)^{\cdot n-1}\cdot(\sigma^*(K_X+\mu_L L)+(c-\mu_Lx)F) \bigr)dx.
\end{equation}
The invariant $\xi_L(Z)$ was originally introduced in \cite{Fuj15} for the case that 
 $L=-K_X$, i.e., $(X,-K_X)$ is a Fano variety.
 Later, this invariant was extended in \cite[Definition $2.8$]{DL23} for the case of any polarized variety $(X,L)$,
 which is called the $\beta$-invariant in the literature of K-stability.

Now, we shall consider the relation between $\xi_L(Z)$ and the Futaki invariant.
For $Z$, let us consider the test configuration $(\mathcal{X},\mathcal{L}_t)$ consisting of the blow up $\eta:\mathcal{X}\to X\times \mathbb{P}^1$ along $Z\times\{0\}$ and $\mathcal{L}_t:=\eta^* p^*L-t E$ for $t\in (0, \ep_L(Z))\cap \Q$.
The space $\mathcal{X}$ is called the degeneration of $X$ to the normal cone of $Z$ in \cite{RT07}.
Here $E$ is the exceptional divisor of the blow-up $\eta$. 
See \cite{RT07} for the details of the notions that appeared here.
As seen in Remark 2.6 in \cite{FY24}, the invariant $\xi_L(Z)$ coincides with the limit of the Donaldson-Futaki invariant $\mathrm{DF}(\mathcal{X}, \mathcal{L}_t)$ as $t$ goes to $\ep_L(Z)$.
 
For the relation among the $\beta$-invariant, the Donaldson-Futaki invariant and $\xi_L(Z)$, they
are equal to each other up to the multiplicative positive constant.
 Precisely, we have 
 \begin{equation}\label{eq:3Invariants}
 \xi_L(Z)=\lim_{t\to \ep_L(Z)} {\mathrm{DF}}(\mathcal X, \mathcal L_t)=\frac{1}{2(n-1)!}\beta_L(Z),
 \end{equation} 
where the first equality follows from \cite[Remark 2.6]{FY24} and the second equality follows from \cite[Proposition 3.9]{DL23}.
 \begin{remark}\label{rem:Fujita}\rm
In order to explain about geometric motivation of the invariant $\xi_L(D)$,
we assume that an $n$-dimensional polarized variety $(X,L)$ is slope semistable along a smooth prime divisor $D$. 
 As it was observed in the proof of \cite[Proposition 2.2]{FY24}, slope semistability implies that
 \begin{align*}
 \frac{1}{2(n-1)!}&\bigl( (L^{\cdot n-1})\cdot(-K_X) \bigr)\int_0^{\ep}\frac{1}{n!}\bigl((L-x D)^{\cdot n}\bigr)dx  \\
 \geqslant&\frac{1}{n!}\left(
 \int_0^\ep\frac{1}{2(n-1)!}\bigl((L-x D)^{\cdot n-1}\cdot(-K_X)\bigr)dx+\frac{1}{2}
 \left[ \frac{1}{n!}\bigl((L-x D)^{\cdot n}\bigr)
 \right]_0^\ep
\right)   \\
\Leftrightarrow \quad
\bigl((L^{\cdot n-1})&\cdot(-K_X)\bigr)\int_0^\ep \bigl((L-x D)^{\cdot n}\bigr)dx 
-(L^{\cdot n})\int_0^\ep \bigl((L-x D)^{\cdot n-1}\cdot(-K_X)\bigr)dx   \\
\hspace{3cm} & \geqslant \frac{(L^{\cdot n})}{n}\bigl((L-\ep D)^{\cdot n} -(L^{\cdot n})\bigr). 
\end{align*}
Plugging the equality
\begin{align*}
\int_0^\ep \bigl((L-x D)^{\cdot n-1}\cdot D\bigr)dx
&= \left[ -\frac{1}{n}\bigl((L-x D)^{\cdot n} \bigr) \right]_0^\ep\\ 
&= \frac{1}{n}\bigl((L^{\cdot n})-(L-\ep D)^{\cdot n} \bigr)
\end{align*}
into the above last inequality, we have
\begin{align}
&\int_0^\ep\Biggl( \bigl(L^{\cdot n-1}\cdot(-K_X)\bigr)\cdot(L-xD)^{\cdot n}
-(L^{\cdot n})\cdot\bigl( (L-xD)^{\cdot n-1}\cdot (-K_X)\bigr) \notag \\
& \hspace{7.6cm} +(L^{\cdot n})\cdot\bigl( (L-xD)^{\cdot n-1}\cdot D\bigr) 
\Biggr) dx \geqslant 0 \notag \\
\Leftrightarrow \quad &
\int_0^\ep \bigl((L-xD)^{\cdot n-1} \bigr)\cdot \bigl( \mu_L(L-xD)+K_X+D \bigr)dx \geqslant 0. \label{ineq:xi_inv}
\end{align}
We remark that our assumption implies that $c=1$ and $\sigma$ is the identity map in $\eqref{def:xi}$.
Thus, we see that the left hand side of $\eqref{ineq:xi_inv}$ equals $\xi_L(D)/n$.
 \end{remark}
From Remark \ref{rem:Fujita} and \cite{RT07}, we can replace the slope semistability along $D$ with the semipositivity  of $\xi_L(D)$.
Then, Fujita and the third author proved that if $\xi_L(D) \geqslant 0$ for any ample $\Q$-line bundle $L$ and any smooth prime divisor $D$ on an $n$-dimensional Bott manifold $X$, then $X\cong (\P^1)^n$.

For our purpose, we assume that $(X,L)$ is an $n$-dimensional smooth polarized toric variety and $D$ is a smooth toric prime divisor on $X$.
Let $P$ be the moment polytope of $(X,L)$, and $\bs v_D \in N_\R$ the ray generator corresponding to $D$.
Let 
\[
\widetilde{\ell}_D(\bs x):= \langle \bs x, \bs v_D\rangle + \widetilde{a}_D
\] 
be the affine function induced by $\bs v_D$ where the constant $\widetilde{a}_D$ is determined so that 
\[
\min \set{\widetilde{\ell}_D(\bs x) | \bs x \in P}=0.
\]
Let $\bs v_D^\sharp$ be the holomorphic vector field induced by $\bs v_D$.
Then, the Futaki invariant of 
$\bs v_D^\sharp$ is equal to the Donaldson-Futaki invariant
\[
\mathrm{Vol}(P) \int_{\partial P} \widetilde{\ell}_D(\bs x) d\sigma
-\mathrm{Vol}(\partial P)\int_{P} \widetilde{\ell}_D(\bs x) dv
\]
of the product test configuration induced by $\bs v_D^\sharp$. 

\begin{lemma}\label{lem:xi_futaki}
Let $(X,L)$ be a smooth polarized toric variety. Let $D$ be a smooth toric prime divisor on $X$.
If the pseudo-effective threshold
\[
\tau_L(D):=\sup\set{t\in\R | \bigl((L-t D)^{\cdot n}\bigr)>0}
\]
is equal to the Seshadri constant $\ep_L(D)$, then the Futaki invariant of $\bs v_D^\sharp$ is equal to $\xi_L(D)$.
\end{lemma}
\begin{proof}
Let $\ell_D(\bs x)$ be the defining affine function of the facet of $P$ corresponding to $D$: 
\[
\ell_D(\bs x):=\langle \bs x, \bs v_D \rangle + a_D.
\]
Let $\mathcal{P}$ be the $(n+1)$-dimensional moment polytope of (the compactification of) the test configuration $(\mathcal{X},\mathcal{L}_t)$ as above.
Then, the polytope $\mathcal{P}$ coincides with 
\[
\mathcal{P}_t=\Big\{(\bs x, h)\in P \times \R \mid\,
		\max\set{\ell_D(\bs x) -t, 0} \le h \le \max \set{\ell_D(\bs x)\mid\, \bs x \in P} \Big\}.
\]
Remark that the level set
\[
\mathcal{P}_t\cap \set{h=0}=
	\set{\bs x \in P | \ell_D(\bs x) -t =0}
\]
is the moment polytope of $(X, L-tD)$ as long as $t\in (0,\ep_L(D))\cap \Q$.

On the other hand, the moment polytope $\mathcal P'$ of (the compactification of) the product test configuration induced by $\bs v_D^\sharp$ is given by
\[
\mathcal{P}'=
\Big\{
(\bs x, h)\in P \times \R 
\mid\,
\widetilde{\ell}_D(\bs x) \leqslant h \leqslant \max \set{\widetilde{\ell}_D(\bs x)\mid\, \bs x \in P}
\Big\}.
\]
The assumption that $\ep_L(D)=\tau_L(D)$ implies $\ell_D=\widetilde{\ell}_D$.
In particular, $\mathcal{P}_{\ep_L(D)}$ coincides $\mathcal{P}'$.
Hence, we have 
\[
F_{X_P}(\bs v_D^\sharp)= \lim_{t\to \ep_L(D)}\mathrm{DF}(\mathcal{X},\mathcal{L}_t)=\xi_L(D).
\]
The proof is completed.
\end{proof}
Firstly, we show that Theorem \ref{thm:main} implies Theorem \ref{thm:FY}. Assume that $X$ is an $n$-dimensional slope semistable Bott manifold whose moment polytope is given by $\eqref{eq:polytope_P}$.
Let $D_{2n-1}$ and $D_{2n}$ be the smooth toric divisors corresponding to $\bs v_{2n-1}$ and $\bs v_{2n}$ respectively. 
Then, the moment polytope of $(X, L-tD_{2n})$ is equal to
\[
P_{\, [-a_{2n}+t, a_{2n-1}]}
=
P_{(-\infty,+\infty)}
\cap
\set{-a_{2n}+t \leqslant x_n \leqslant a_{2n-1}}.
\] 
Since we have
\[
\ep_L(D_{2n})=\tau_L(D_{2n})=a_{2n-1}+a_{2n},
\]
Lemma \ref{lem:xi_futaki} implies that
\[
	F_{X}(\bs v_{D_{2n}}^{\sharp})
	=
	F_{X}(x_{n})
	=
	\xi_L(D_{2n}) 
	\geqslant
	0,
\]
where the last inequality follows from the assumption of slope semistability. 
Similarly, we have
$$
	F_{X}(\bs v_{D_{2n-1}}^{\sharp})
	=
	F_{X_P}(-x_{n})
	=
	\xi_L(D_{2n-1}) 
	\geqslant
	0,
$$
because the moment polytope of $(X, L-tD_{2n-1})$ is equal to
\[
	P_{\, [-a_{2n}, a_{2n-1}-t]}
	=
	P_{(-\infty,+\infty)}
	\cap
	\set{-a_{2n} \leqslant x_n \leqslant a_{2n-1}-t }.
\] 
Therefore, we have
$$
	F_{X_P}(x_n)=0.
$$
Now, Corollary \ref{cor:splitting} implies the splitting $X=X'\times \P^1$ where $X'$ is the Bott manifold of dimension $n-1$.
Applying the above argument to $X'$ inductively as in Section \ref{sec:induction}, we find that $X\cong (\P^1)^n$.

Finally, we will prove the converse: Theorem \ref{thm:FY} implies Theorem \ref{thm:main}.
Let us take the divisor $D$ as in Lemma 3.1 in \cite{FY24}.
Then, Lemma 3.3 (2) in \cite{FY24} assures that
\[
	\ep_L(D)=\tau_L(D)=1.
\]
Under the above condition, we have
\begin{equation}\label{eq:xi_beta}
	\xi_L(D) 
	=
	\frac{1}{2(n-1)!}\beta_L(D)
\end{equation}
by $\eqref{eq:3Invariants}$. From \cite[$(5.7)$]{DL23}, we have the formula of $\beta_L(D)$ such that
\begin{equation}\label{eq:beta}
\beta_L(D)=(A_X(D)-1)\Vol(P)+\Vol(\p P)\braket{b_P-b_{\p P}, \bs v_D},
\end{equation}
where $A_X(D)$ is the log discrepancy of $D$, $b_P$ and $b_{\p P}$ are the barycenters of $P$ and $\p P$ respectively. Then $\eqref{eq:xi_beta}$ and $\eqref{eq:beta}$ imply that
\[
\xi_L(D)=\frac{1}{2(n-1)!} \beta_L(D)=\frac{1}{2(n-1)!}(A_X(D)-1)\mathrm{Vol}(P)\geqslant 0.
\]
Hence, we can apply Proposition 4.2 in \cite{FY24} to $(X,D)$.
Repeating the above argument inductively, we get the conclusion.

\bibliographystyle{amsalpha}

\end{document}